\documentclass[draft]{birkmult}

\usepackage{amsmath,amssymb,amsthm}
\usepackage{enumerate}
\usepackage{graphics,color}
\usepackage{bbding}
\usepackage{pifont,mathrsfs}

\newtheorem{theorem}{Theorem}[section]
\newtheorem{lemma}[theorem]{Lemma}

\theoremstyle{definition}
\newtheorem{assumption}[theorem]{Assumption}
\newtheorem{remark}[theorem]{Remark}

\def\dd{\mathrm{d}}
\def\ee{\mathrm{e}}

\def\Re{\mathrm{Re}\hspace*{0.5mm}}

\def\RR{{\mathbb{R}}}
\def\NN{{\mathbb{N}}}

\def\CC{\mathbb{C}}

\def\LLL{\mathscr{L}}
\def\veps{\varepsilon}

\begin{document}


\title[Splitting with approximations]{Operator splitting with spatial-temporal disc\-reti\-zation}

\author[A. B\'{a}tkai]{Andr\'{a}s B\'{a}tkai}
\address{E\"{o}tv\"{o}s Lor\'{a}nd University, Institute of Mathematics, H-1117 Budapest, P\'{a}zm\'{a}ny P. s\'{e}t\'{a}ny 1/C, Hungary.}
\email{batka@cs.elte.hu}
\thanks{Research partially supported by the Alexander von Humboldt-Stiftung.}

\author[P. Csom\'{o}s]{Petra Csom\'{o}s}
\address{Leopold--Franzens--Universit\"{a}t Innsbruck, Institut f\"{u}r Mathematik, Technikerstra{\ss}e 13, A-6020 Innsbruck, Austria}
\email{petra.csomos@uibk.ac.at}

\author[B. Farkas]{B\'{a}lint Farkas}
\address{Technische Universit\"{a}t Darmstadt, Fachbereich Mathematik, Schlo{\ss}gartenstra{\ss}e 7, D-64289 Darmstadt, Germany}
\email{farkas@mathematik.tu-darmstadt.de}

\author[G. Nickel]{Gregor Nickel}
\address{Universit\"at Siegen, FB 6 Mathematik, Walter-Flex-Stra{\ss}e 3, D-57068 Siegen, Germany.}
\email{nickel@mathematik.uni-siegen.de}


\date\today

\begin{abstract}
Continuing earlier investigations, we analyze the convergence of operator splitting procedures combined with spatial discretization and rational approximations.
\end{abstract}
\maketitle


\section{Introduction}

Operator splitting procedures are special finite difference methods one uses to solve partial differential equations numerically. They are certain time-discretization methods which simplify or even make the numerical treatment of differential equations possible.

The idea is the following. Usually, a certain physical phenomenon is the combined effect of several processes. The behaviour of a physical quantity is described by a partial differential equation in which the local time derivative depends on the sum of the sub-operators corresponding to the different processes. These sub-operators  are usually of different nature. For each sub-problem corresponding to each sub-operator there might be a fast numerical method providing accurate solutions. For the sum of these sub-operators, however, we usually cannot find an adequate method. Hence, application of operator splitting procedures means that instead of the sum we treat the sub-operators separately. The solution of the original problem is then obtained from the numerical solutions of the sub-problems. For a more detailed introduction and further references, see the monographs by Hairer et al. \cite{HLW}, Farag\'{o} and Havasi \cite{Farago-Havasi_book}, Holden et al. \cite{Holden-Karlsen-Lie-Risebro}, or Hunsdorfer and Verwer \cite{Hundsdorfer-Verwer}.

Since operator splittings are time-discratization methods, the analysis of their convergence plays an important role. In our earlier investigations in B\'{a}tkai, Csom\'{o}s, Nickel \cite{Batkai-Csomos-Nickel} we achieved theoretical convergence analysis of problems when operator splittings were applied together with some spatial approximation scheme. In the present paper we additionally treat temporal discretization methods of special form. Since rational approximations often occur in practice (consider e.g. Euler and Runge\,--\,Kutta methods, or any linear multistep method), we will concentrate on them. Let us start by setting the abstract stage. 

\begin{assumption}\label{gen_ass_abc}
Assume that $X$ is a Banach space, $A$ and $B$ are closed, densely defined linear operators generating the strongly continuous operator semigroups $(T(t))_{t\ge 0}$ and $(S(t))_{t\ge 0}$, respectively. Further, we suppose that the closure $\overline{A+B}$ of $A+B$ with $D(\overline{A+B})\supset D(A)\cap D(B)$ is also the generator of a strongly continuous semigroup $(U(t))_{t\ge 0}$.
\label{as:genass}
\end{assumption}

For the terminology and notations about strongly continuous operator semigroups see the monographs by Arendt et al. \cite{ar-ba-hi-ne} or Engel and Nagel \cite{Engel-Nagel}.
Then we consider the following abstract Cauchy problem

\begin{equation}
\left\{
\begin{aligned}
\frac{\dd u(t)}{\dd t}&=(A+B)u(t), \qquad t\ge 0, \\
u(0)&=x\in X.
\label{acpspl}
\end{aligned}
\right.
\end{equation}

For the different splitting procedures the \emph{exact split} solution of problem \eqref{acpspl} at time $t\ge 0$ is given by
\begin{align*}
u^\text{sq}_n(t)&:=[S(t/n)T(t/n)]^nx \quad \text{(sequential)},\\
u^\text{St}_n(t)&:=[T(t/2n)S(t/n)T(t/2n)]^nx \quad \text{(Strang)},\\
u^\text{w}_n(t)&:=[\Theta S(t/n)T(t/n)+(1-\Theta)T(t/n)S(t/n)]^n x \text{ with } \Theta\in(0,1) \quad \text{(weighted)}.
\end{align*}


\medskip \noindent However, in practice, we obtain the numerical solution of the problem \eqref{acpspl} by
\renewcommand{\labelitemi}{{\ding{100}}}
\begin{itemize}
\item applying a \emph{splitting procedure} with operator $A$ and $B$,
\item defining a mesh on which the split problems should be \emph{discretized in space}, and
\item using a certain \emph{temporal approximation} to solve these (semi-)discretized equations.
\end{itemize}
Thus, the properties of this complex numerical scheme should be investigated. In order to work in an abstract framework, we introduce the following spaces and operators, see Ito and Kappel \cite[Chapter 4]{Ito-Kappel}.
\begin{assumption}\label{ass:JmPm}
Let $X_m$, $m\in\mathbb{N}$ be Banach spaces and take operators
\begin{equation}
P_m:\ X\rightarrow X_m \qquad \mbox{and} \qquad J_m:\ X_m\rightarrow X
\nonumber
\end{equation}
having the following properties:
\renewcommand{\labelenumi}{(\roman{enumi})}
\begin{enumerate}
\item $P_mJ_m=I_m$ for all $m\in\mathbb{N}$, where $I_m$ is the identity operator in $X_m$,
\item $\lim\limits_{m\rightarrow\infty}J_mP_mx=x$ for all $x\in X$,
\item $\|J_m\|\le K$ and $\|P_m\|\le K$ for all $m\in\mathbb{N}$ and some given constant $K>0$.
\end{enumerate}
\end{assumption}
The operators $P_m$, $m\in\mathbb{N}$ are usually some projections onto the spatial ``mesh'' $X_m$, while the operators $J_m$ correspond to the interpolation method resulting the solution in the space $X$, but also Fourier--Glerkin methods fit in this
framework.

Let us recall the following definitions and results from \cite{Batkai-Csomos-Nickel}.
 First we assume that the exact solution $u$ of problem \eqref{acpspl} is obtained by using only a splitting procedure and discretization in space.
\begin{assumption}\label{ass:StabCons}
Let $(A_m,D(A_m))$ and $(B_m,D(B_m))$, $m\in\mathbb{N}$, be operators on $X_m$ and let $(A,D(A))$ and $(B,D(B))$ be operators on $X$, that satisfy the following:
\renewcommand{\labelenumi}{(\roman{enumi})}
\renewcommand{\labelenumii}{(\alph{enumii})}
\begin{enumerate}
\item \emph{Stability:} \\
Assume that there exist constants $M \ge 0$ such that
\begin{enumerate}
\item $\|\Re\lambda R(\lambda,A_m)\|\le M$ and $\|\Re\lambda R(\lambda,A)\|\le M$,
\item $\|\Re\lambda R(\lambda,B_m)\|\le M$ and $\|\Re\lambda R(\lambda,B)\|\le M$
\end{enumerate} for all $\Re\lambda>0$.
\item \emph{Consistency:}\label{ass:consistency}
Assume that $P_m D(A)\subset D(A_m)$, $P_m D(B)\subset D(B_m)$, and
\begin{enumerate}
\item $\lim\limits_{m\to\infty}J_mA_mP_mx=Ax$ \qquad for all $x\in D(A)$,
\item $\lim\limits_{m\to\infty}J_mB_mP_mx=Bx$ \qquad for all $x\in D(B)$.
\end{enumerate}
\end{enumerate}

%
\end{assumption}
\begin{remark}\label{rem:conv}
\begin{enumerate}[a)]
\item If Assumption \ref{ass:StabCons} is satisfied for $M=1$, then by the Hille--Yosida Theorem $A,B,A_m,B_m$ are all generators of contraction semigroups  $(T(t))_{t\geq 0}$, $(S(t))_{t\geq 0}$, $(T_m(t))_{t\geq 0}$, $(S_m(t))_{t\geq 0}$. Furthermore, from the Trotter\,--\,Kato Approximation Theorem (see Ito and Kappel \cite[Theorem 2.1]{Ito-Kappel2}) it follows that the approximating semigroups converge to the original semigroups locally uniformly, that is: \\
\noindent \emph{Convergence:}
\renewcommand{\labelenumi}{(\alph{enumi})}
\begin{enumerate}
\item $\lim\limits_{m\to\infty}J_mT_m(h)P_mx=T(h)x$ \qquad $\forall x\in X$ and uniformly for $h\in[0,t_0]$,
\item $\lim\limits_{m\to\infty}J_mS_m(h)P_mx=S(h)x$ \qquad $\forall x\in X$ and uniformly for $h\in[0,t_0]$
\end{enumerate}
for any $t_0\geq 0$.
\item In turn, the resolvent estimates are  satisfied if $A,B,A_m,B_m$ are all generators of  bounded semigroups, with the same bound for all $m\in\NN$.
 One may even assume that these semigroups have the same exponential estimate, that is
 $$
 \|T_m(t)\|,\:\|T(t)\|,\:\|S_m(t)\|,\|S(t)\|\leq M\ee^{\omega t}\quad\mbox{for all $t\geq0 $}.
 $$
 In the following this would result in a simple rescaling that we want to spare for the sake of brevity.
\end{enumerate}
\end{remark}

In order to prove the convergence of the \emph{splitting} procedures in this case, we need to formulate a modified version of Chernoff's Theorem being valid also for the spatial discretizations.  Our main technical tool will be the following theorem, whose proof can be carried out along the same lines as Theorem 3.12 in \cite{Batkai-Csomos-Nickel}. Let us agree on the following terminology. We say that for a sequence $a_{m,n}$ the limit $$\lim\limits_{m,n\to\infty}a_{m,n}=:a$$ exists if for all $\varepsilon>0$ there exists $N\in\mathbb{N}$ such that for all $n,m\ge N$ we have $\left\|a_{m,n}-a\right\|\le\varepsilon$.

\begin{theorem}[Modified Chernoff--Theorem, {\cite[Theorem 3.12]{Batkai-Csomos-Nickel}}]\label{thm:mod}
Consider  a sequence of functions $F_m:\RR^+\rightarrow\LLL(X_m)$, $m\in\mathbb{N}$, satisfying
\begin{equation}
F_m(0)=I_m \qquad \mbox{for all } m\in\mathbb{N},
\label{chernoff2-1}
\end{equation}
and that there exist constants $M\ge 1$, $\omega\in\RR$, such that
\begin{equation}
\|[F_m(t)]^k\|_{\LLL(X_m)}\le M\mathrm{e}^{k\omega t} \qquad \mbox{for all } t\ge 0,\ m,k\in\mathbb{N}.
\label{chernoff2-2}
\end{equation}
Assume further that
\begin{equation}
\exists \lim\limits_{h\rightarrow 0}\dfrac{J_mF_m(h)P_mx-J_m P_m x}{h}
\label{chernoff2-tfh}
\end{equation}
uniformly in $m\in\NN$, and that
\begin{equation}
Gx:= \lim\limits_{m\rightarrow\infty}\lim\limits_{h\to 0}\dfrac{J_mF_m(h)P_mx-J_m P_m x}{h}
\label{chernoff2-3}
\end{equation}
exists for all $x\in D\subset X$, where $D$ and $(\lambda_0-G)D$ are dense subspaces in $X$ for $\lambda_0>0$. Then the closure $\overline G$ of $G$ generates a bounded strongly continuous semigroup $\left( U(t) \right)_{t\ge 0}$, which is given by
\begin{equation}
U(t)x=\lim\limits_{m,n\rightarrow\infty} J_m[F_m(\tfrac{t}{n})]^n P_m x
\label{convergence2}
\end{equation}
for all $x\in X$ uniformly for $t$ in compact intervals.
\label{thm:chernoff2}
\end{theorem}


\section{Rational approximations}
Our aim is to show the convergence of various splitting methods when combined with both spatial and temporal disretization.
As temporal discretizations we consider finite difference methods, or more precisely rational approximations of the exponential function. Throughout this section, we suppose that $r$ and $q$ will be rational functions \textbf{approximating the exponential function} at least of order one, that is we suppose
$$r(0)=r'(0)=1\quad\mbox{and}\quad q(0)=q'(0)=1.$$
Further, we suppose that these functions are bounded on the closed left half-plane
$$\CC_-:=\bigl\{z\in \CC:\Re z\leq 0\bigr\}.$$
Rational (e.g.~$A$-stable) functions typically appearing in numerical analysis satisfy these conditions. An important consequence of the boundedness in the closed left half-plane is that the poles of $r$ have strictly positive real part, and thus lie in some sector
$$
\Sigma_{\theta}:=\bigl\{z:z\in \CC,\: |\arg (z)|<\theta\bigr\}
$$
of opening half-angle $\theta\in [0,\frac\pi2)$.

\medskip It is clear that for an application of the Modified Chernoff--Theorem \ref{thm:mod} uniform convergence (w.r.t.~$m$ or $h$) plays a crucial role here (cf.~ \cite{Batkai-Csomos-Nickel}).  Hence, the following lemma will be the main technical tool in our investigations.
\begin{lemma}\label{lem:key}
Let  $A,A_m,P_m,J_m$ be as in Assumptions \ref{ass:JmPm} and \ref{ass:StabCons}.
Let  $r$ be a rational approximation of the exponential being bounded on the closed left half-plane $\CC_-$.
Then  for all $x\in D(A)$ we have
\begin{equation}\label{eq:elso}
\left\| \frac{J_m r(hA_m)P_mx-J_mP_mx}{h}-J_mA_mP_mx\right\| \to 0
\end{equation}
uniformly in $m\in\NN$ for $h\to 0$ .
\end{lemma}
The proof of this lemma is postponed to the end of this section. With its help, however, one can prove
the next results: (1) on convergence of spatial-temporal discretization without splitting, (2) on convergence of the splitting procedures combined with spatial and temporal approximations.
\begin{theorem}\label{thm:nosplit}
Let $A,A_m,P_m,J_m$ be as in Assumptions \ref{ass:JmPm} and \ref{ass:StabCons} and let  $A$ generate the semigroup $(T(t))_{t\geq 0}$. Suppose that $r$ is a rational function approximating the exponential function bounded on $\CC_-$, and that there exist constants $M\ge 1$ and $\omega\in\RR$ with
\begin{equation}\label{ass:rstab}
\|[r(hA_m)]^k\|\le M\ee^{k\omega h} \quad \text{ for all } h\ge 0,\ k,m\in\mathbb{N}.
\end{equation}
Then
\begin{equation*}
\lim_{m,n\to\infty} J_m r(\tfrac{t}{n}A_m)^nP_m x = T(t)x,
\end{equation*}
uniformly for $t\geq 0$ in compact intervals.
\end{theorem}
\begin{proof}
We will apply Theorem \ref{thm:chernoff2} with $F_m(h)=r(hA_m)$. The stability criteria \eqref{chernoff2-1}-\eqref{chernoff2-2} follow directly from $r(0)=1$ and assumption \eqref{ass:rstab}. For the consistency \eqref{chernoff2-3} we have to show the existence of the limit in \eqref{chernoff2-tfh} uniformly in $m\in\NN$. But this is exactly the statement of Lemma \ref{lem:key}.
\end{proof}

Here is the announced theorem on the convergence of the sequential splitting with spatial and rational temporal discretization.
\begin{theorem}\label{thm:sq_conv}
Let $A,B,A_m,B_m,P_m,J_m$ be as in  Assumption \ref{gen_ass_abc}, Assumptions \ref{ass:JmPm} and \ref{ass:StabCons}, and let $(U(t))_{t\geq 0}$ denote the semigroup generated by the closure of $A+B$. Suppose that the following stability condition is satisfied:
\begin{equation}
\bigl\|[ q(h B_m)r(h A_m)]^k\bigr\|\le M\mathrm{e}^{kh\omega} \quad \mbox{for all } h\ge 0,\ k,m\in\mathbb{N}.
\label{stab_approx}
\end{equation}%
Then the sequential splitting is convergent, i.e.,
\begin{equation*}
\lim_{m,n\to\infty}  [q(\tfrac{t}{n} B_m)r(\tfrac{t}{n} A_m)]^n x = U(t)x.
\end{equation*}
\end{theorem}
\begin{proof}
We apply Theorem \ref{thm:chernoff2} with the choice  $F_m(t) = q(tB_m)r(tA_m)$ for an arbitrarily fixed $t\ge 0$. Since stability is assumed, we only have to check the consistency. To do that, first we have to show that
\begin{equation}\label{eq:toshow}
\lim\limits_{h\to 0}\frac{J_m q(h B_m)r(h A_m)P_mx-J_mP_mx}{h}=J_m(A_m+B_m)P_m x
\end{equation}
for all $x\in D(A+B)$ and uniformly for $m\in\NN$. The left-hand side of \eqref{eq:toshow} can be written as:
\begin{align*}
& \frac{J_m q(h B_m)r(h A_m)P_mx-J_mP_mx}{h} \\
&\quad= J_m q(h B_m)P_m \frac{J_m r(h A_m)P_m x- J_mP_m x}{h} + \frac{J_mq(h B_m)P_m x -J_mP_mx}{h}.
\end{align*}
Since the topology of pointwise convergence on a dense subset of $X$ and the topology of uniform convergence on relatively compact subsets of $X$ coincide on bounded subsets of $\LLL(X)$ (see e.g., Engel and Nagel \cite[Proposition A.3]{Engel-Nagel}), it follows from Lemma \ref{lem:key} that the expression above converges uniformly to $J_m(A_m+B_m)P_m x$. 
\end{proof}

\begin{theorem}
Suppose that the conditions of Theorem \ref{thm:sq_conv} are satisfied, but replace the stability assumption with either
\begin{equation*}
\|[ r(\tfrac{h}{2} A_m)q(hB_m)r(\tfrac{h}{2} A_m)]^k\|\le M\mathrm{e}^{kh\omega} \quad \mbox{for all } h\ge 0,\ k,m\in\mathbb{N}
\end{equation*}
for the Strang splitting, or
\begin{equation*}
\| \left[\Theta  q(hB_m)r(hA_m)+ (1-\Theta) r(hA_m)q(h B_m)\right]^k\|\le Me^{kh\omega}
\end{equation*}
for a $\Theta\in[0,1]$ and for all $h\ge 0,\ k,m\in\mathbb{N}$ in case of the weighted splitting. Then the Strang and weighted splittings, respectively, are convergent, i.e.,
\begin{align*}
\lim_{n,m\to\infty} \left[ r(\tfrac{t}{2n} A_m)q(\tfrac{t}{n} B_m)r(\tfrac{t}{2n} A_m)\right]^n x &= U(t)x \quad \text{(Strang)}, \\
\lim_{n,m\to\infty} \left[\Theta  q(\tfrac{t}{n} B_m)r(\tfrac{t}{n} A_m)+ (1-\Theta) r(\tfrac{t}{n} A_m)q(\tfrac{t}{n} B_m)\right]^n &= U(t)x \quad \text{(weighted)}.
\end{align*}
\end{theorem}
\begin{proof}
The proof is very similar as it was in the case of the sequential splitting in Theorem \ref{thm:sq_conv}. The only difference occurs in formula \eqref{eq:toshow}. In the case of Strang splitting we take
\begin{equation*}
F_m(t)=r(\tfrac{t}{2n} A_m)q(\tfrac{t}{n} B_m)r(\tfrac{t}{2n} A_m)
\end{equation*}
and write
\begin{align*}
& \frac{J_mr(\tfrac{h}{2} A_m)q(hB_m)r(\tfrac{h}{2} A_m)P_mx-J_mP_mx}{h} \\
&\quad= J_mr(\tfrac{h}{2} A_m)q(hB_m)P_m\frac{J_mr(\tfrac{h}{2} A_m)P_mx-J_mP_mx}{h} \\
&\quad\quad+ J_mr(\tfrac{h}{2} A_m)P_m\frac{J_mq(hB_m)P_mx-J_mP_mx}{h}
+\frac{J_mr(\tfrac{h}{2} A_m)P_mx-J_mP_mx}{h}.
\end{align*}
By Lemma \ref{lem:key} this converges uniformly to
\begin{equation*}
J_m(\tfrac{1}{2}A_m+B_m+\tfrac{1}{2}A_m)P_mx=J_m(A_m+B_m)P_mx.
\end{equation*}
For the weighted splitting we choose
\begin{equation*}
F_m(t)=\Theta q(tB_m)r(tA_m) + (1-\Theta)r(tA_m)q(tB_m),
\end{equation*}
which results in
\begin{align*}
& \frac{J_m [\Theta q(tB_m)r(tA_m) + (1-\Theta)r(tA_m)q(tB_m)] P_mx-J_mP_mx}{h} \\
&\quad= \Theta \frac{J_m q(h B_m)r(h A_m)P_mx-J_mP_mx}{h} \\
&\quad\quad +(1-\Theta)\frac{J_m r(h A_m)q(h B_m)P_mx-J_mP_mx}{h}.
\end{align*}
By using the argumentation for sequential splitting, this converges uniformly (in $m$) to
\begin{equation*}
\Theta J_m(A_m+B_m)P_mx+(1-\Theta)J_m(B_m+A_m)P_mx=J_m(A_m+B_m)P_mx
\end{equation*}
as $h\to 0$.
\end{proof}


\section*{Proof of Lemma \ref{lem:key}}
The proof consists of three steps, the first being the case of the simplest possible rational approximation, which describes the backward Euler scheme. The next two steps generalize to more complicated cases.

\noindent {\it Step 1.} Consider first the rational function $r(z)=\frac{1}{1-z}$. Then $r(hA_m)=\tfrac{1}{h} R(\tfrac{1}{h},A_m)$ for $h>0$ sufficiently small. Then the left hand side of  \eqref{eq:elso} takes the form
\begin{align*}
&\left\|\tfrac{1}{h}\left(\tfrac{1}{h} J_m R(\tfrac{1}{h},A_m) P_mx - J_mP_mx\right) - J_mA_mP_mx\right\| = \\
&\qquad= \left\|\tfrac{1}{h}\left(\tfrac{1}{h} J_m R(\tfrac{1}{h},A_m) P_mx - J_m \left(\tfrac{1}{h}-A_m\right) R(\tfrac{1}{h},A_m) P_mx \right) - J_mA_mP_mx\right\| = \\
&\qquad= \left\| J_m \left( \tfrac{1}{h}R(\tfrac{1}{h},A_m)-I_m \right) A_mP_mx\right\|.
\end{align*}
For the uniformity of the convergence in $m\in\NN$, take  $x\in D(A)$. Then for $\lambda>0$
\begin{equation*}
\lambda J_m R(\lambda,A_m)P_m x = J_m R(\lambda,A_m)A_m P_m x + J_mP_m x.
\end{equation*}
Hence,
\begin{equation*}
\|\lambda J_m R(\lambda,A_m)P_m x -J_mP_mx\| \leq \|J_mR(\lambda, A_m)P_m\|\cdot \|J_mA_mP_mx\|
\end{equation*}
follows. By the stability in Assumption \ref{ass:StabCons},
\begin{equation*}
\|J_mR(\lambda, A_m)P_m\| \leq \frac{K^2M}{\lambda}\quad \mbox{holds for $\lambda>0$}.
\end{equation*}
Further, by the consistency in Assumption \ref{ass:StabCons} the sequence  $J_mA_mP_mx$ is bounded. Therefore
\begin{equation} \label{eq:lambarezolv}
\lambda J_m R(\lambda,A_m)P_m x \to J_mP_mx
\end{equation}
as $\lambda\to\infty$ uniformly in $m\in\NN$.  Since $\|J_mR(\lambda, A_m)P_m\|$ is uniformly bounded in $m\in\NN$, we obtain by the densness of $P_mD(A)\subset X_m$ that \eqref{eq:lambarezolv} holds even for arbitrary $x\in X$.

Since for $x\in D(A)$, by the consistency in Assumption \ref{ass:StabCons}, the set
\begin{equation*}
\{J_mA_mP_m x\,:\,m\in\NN\}\cup\{Ax\}
\end{equation*}
is compact, for arbitrary $\veps>0$ there is $N\in \NN$ such that the balls $B(J_iA_iP_i x,\veps)$ for $i=1,\dots N$ cover this compact set.
Now let $m\in\NN$ arbitrary and pick $i\leq N$ with $\|J_iA_iP_i x-J_mA_mP_m x\|\leq \veps$. Then we can write
\begin{align*}
&\left\|\tfrac{1}{h}J_mR(\tfrac{1}{h},A_m)A_mP_mx-J_mA_mP_m x\right\|\\
\leq & \left\|\tfrac{1}{h}J_mR(\tfrac{1}{h},A_m)A_mP_mx-J_iA_iP_i x\right\|+\|J_mA_mP_mx-J_iA_iP_i x\|\\
\leq & \left\|\tfrac{1}{h}J_mR(\tfrac{1}{h},A_m)P_m(J_mA_mP_mx-J_iA_iP_ix)\right\|\\
+ & \left\|(\tfrac{1}{h}J_mR(\tfrac{1}{h},A_m)P_m-I_m)J_iA_iP_ix\right\|+\veps\\
\leq & C\veps +\left\|(\tfrac{1}{h}J_mR(\tfrac{1}{h},A_m)P_m-I_m)J_iA_iP_ix\right\|+\veps,
\end{align*}
with an absolute constant $C\geq 0$ being independent on $m\in\NN$ for $h$ sufficiently small. Because of \eqref{eq:lambarezolv}, the term in the middle also converges to 0 as $h\to 0$ (choosing $\lambda=\tfrac{1}{h}$). This proves the validity of \eqref{eq:elso} for our particular choice of the rational function $r$.

\medskip
\noindent {\it Step 2.}
Next, let $k\in \NN$ and $r(z):=\frac{1}{(1-z/k)^k}$. Then $r(0)=1$, $r'(0)=1$ and $r(hA_m)=[\frac kh R(\frac kh,A_m)]^k$. We have to prove
\begin{equation}\label{eq:1}
\tfrac{1}{h}\Bigl[J_m(\tfrac kh R(\tfrac kh,A_m))^kP_mx-J_mP_mx\Bigr]-J_mA_mP_mx\to 0
\end{equation}
uniformly for $m\in\NN$ as $h\to 0$. To do this we shall repeatedly use the following ``trick'': for $y\in D(A_m)$ we have $y=R(\frac kh,A_m)(\frac kh-A_m)y$. Hence we obtain
\begin{align*}
J_mP_m x&=J_m\tfrac kh R(\tfrac kh ,A_m)P_m x- J_m R(\tfrac kh, A_m)A_mP_m x\\
&=\cdots =J_m[\tfrac kh R(\tfrac kh ,A_m)]^kP_m x-\sum_{j=0}^{k-1}J_m[\tfrac kh R(\tfrac kh ,A_m)]^j R(\tfrac kh ,A_m) A_mP_m x.
\end{align*}
By inserting this into the left hand side of \eqref{eq:1} we get
\begin{align}
\notag&\tfrac{1}{h}\Bigl[J_m(\tfrac kh R(\tfrac kh,A_m))^kP_mx-J_mP_mx\Bigr]-J_mA_mP_mx\\
\label{eq:2}&\qquad\qquad=\frac 1k\sum_{j=1}^{k}J_m[\tfrac kh R(\tfrac kh ,A_m)]^j A_mP_m x-J_mA_mP_mx.
\end{align}
By Step 1, we have $J_m \tfrac kh R(\tfrac kh ,A_m)]  A_mP_m x\to J_mA_mP_mx$ uniformly in $m\in\NN$ as $h\to 0$. Now, since $J_m [\tfrac kh R(\tfrac kh ,A_m)]^jP_m$ is uniformly bounded for $m\in\NN$, $1\leq j\leq k$ and $h$ small enough, we also have
that  $J_m[\tfrac kh R(\tfrac kh ,A_m)]^j A_mP_m x\to J_mA_mP_mx$ uniformly in $m\in\NN$ as $h\to 0$ for each $j=1,\dots k$. This shows that the expression in \eqref{eq:2}
 converges to $0$ uniformly in $m\in \NN$.

\medskip
\noindent {\it Step 3.}
To finish the proof for the case of a general rational function
\begin{equation*}
r(z) =  \frac{a_0+a_1z+\ldots +a_kz^k}{b_0+b_1z+\ldots +b_nz^n}
\end{equation*}
we use the partial fraction decomposition, i.e., we write
\begin{equation*}
r(z)=\sum_{i=1}^l\sum_{j=1}^{\nu_i}\frac{C_{ij}}{(1-z/{\lambda_i})^j},
\end{equation*}
with some uniquely determined $C_{ij}\in\CC$.  Since, by assumption, $r(0)=1$ and $r'(0)=1$, we obtain
\begin{equation}\label{eq:cij}
\sum_{i=1}^l\sum_{j=1}^{\nu_i}C_{ij}=1,\quad \mbox{and}\quad \sum_{i=1}^l\sum_{j=1}^{\nu_i}\frac{j}{\lambda_i}C_{ij}=1.
\end{equation}
Since $r$ is bounded on the left half-plane, we have that the poles $\lambda_i$ of $r$ have positive real part, $\Re\lambda_i>0$.
For $j=1,\dots,\nu_i,\:i=1,\dots,l$ consider the rational functions
$$
r_{ij}(z):=\frac{1}{(1-z/j)^j},
$$
and the operators $A_{ij,m}:=\frac{j}{\lambda_i}A_m$. Then $$r(z)=\sum_{i=1}^l\sum_{j=1}^{\nu_i}C_{ij} r_{ij}(\tfrac{j}{\lambda_i}z).$$
We shall apply Step 2 to these rational functions and to these operators. To do that we have to check if the required assumptions are satisfied. Obviously, $r_{ij}$ have the  properties needed. The consistency part of Assumption \ref{ass:StabCons} is trivially satisfied for $A_{ij,m}$, $P_m$ and $J_m$, $m\in \NN$. The uniform boundedness of $\lambda R(\lambda,A_{ij,m})$  follows from the stability Assumption \ref{ass:StabCons}.(a). Indeed, that condition implies
$$
\|\lambda R(\lambda,A_m\|\leq M_\phi\quad\mbox{for all $\lambda\in\Sigma_\phi$},
$$
where $\Sigma_\phi$ is any sector with opening half-angle  $\phi\in[0,\frac\pi2)$. If $\lambda_i\in\Sigma_\varphi$ is a pole of $r$ then $\tfrac{\lambda_i\lambda}{j}\in \Sigma_{\theta}$ for $\lambda>0$. So have
$$
\|\lambda R(\lambda,A_{ij,m})\|=\|\lambda R(\lambda,\tfrac{j}{\lambda_i}A_m)\|=\|\tfrac{\lambda\lambda_i}{j} R(\tfrac{\lambda\lambda_i}{j},A_m)\|\leq M_\theta\quad\mbox{for all $\lambda>0$.}
$$
By Step 2, we have that
$$
\tfrac{1}{h}\Bigl[J_m\bigl(\tfrac jh R(\tfrac jh,A_{ij,m})\bigr )^jP_mx-J_mP_mx\Bigr]-J_m A_{ij,m} P_mx\to 0
$$
uniformly in $m\in \NN$ as $h\to 0$. By taking also the equalities \eqref{eq:cij} into account this yields
\begin{align*}
&\frac{1}{h}\Bigl[J_mr(hA_m)P_mx-J_mP_mx\Bigr]-J_mA_mP_mx\\
&\quad =\frac{1}{h} \Bigl[\sum_{i=1}^l\sum_{j=1}^{\nu_i} C_{ij}\Bigl(J_m r_{ij}(h \tfrac{j}{\lambda_i}A)P_mx- J_m P_mx\Bigr) \Bigr]-\sum_{i=1}^l\sum_{j=1}^{\nu_i} C_{ij} J_m A_{ij,m}P_mx\\
&\quad =\sum_{i=1}^l\sum_{j=1}^{\nu_i} C_{ij}\Bigl[\frac{1}{h}\Bigl(J_m r_{ij}(h \tfrac{j}{\lambda_i}A)P_mx- J_m P_mx\Bigr) -J_m A_{ij,m}P_mx\Bigr]\to 0
\end{align*}
uniformly in $m\in \NN$ as $h\to 0$. This finishes the proof.
\qed\medskip

Finally we remark that in the present paper we only treated an autonomous evolution equation \eqref{acpspl}. In the case of time-dependent operators $A(t)$ and $B(t)$ we have already shown the convergence in \cite{Batkai-Csomos-Farkas-Nickel} for numerical methods applying splitting and spatial discretization together. The extension of our present results concerning the application of an approximation in time as well, will be the subject of forthcoming work.


\section*{Acknowledgments}
A.~B\'atkai was supported by the Alexander von Humboldt-Stiftung and by the OTKA grant Nr. K81403. The European Union and the European Social Fund have provided financial support to the project under the grant agreement no. T\'{A}MOP-4.2.1/B-09/1/KMR-2010-0003.


\bibliographystyle{plain}


\end{document}